\documentclass{amsart}

\usepackage{amsmath,amssymb,amsbsy,amsthm}
\usepackage{eucal,enumerate}
\usepackage{latexsym}

\begin{document}

\title[Natural connection with totally
skew-symmetric torsion]{NATURAL CONNECTION WITH TOTALLY
SKEW-SYMMETRIC TORSION ON
ALMOST CONTACT MANIFOLDS WITH B-METRIC%
}

\author{MANCHO MANEV
}

\address{       University of Plovdiv,
                Faculty of Mathematics and Informatics,
                Department of Geometry,
                236 Bulgaria Blvd\\
                Plovdiv 4003,
                Bulgaria
} \email{mmanev@uni-plovdiv.bg}

\renewcommand\theenumi{(\roman{enumi})}
\newcommand{\f}{\varphi}
\newcommand{\n}{\nabla}
\newcommand{\lm}{\lambda}
\newcommand{\R}{\mathbb{R}}
\newcommand{\F}{\mathcal{F}}
\newcommand{\W}{\mathcal{W}}
\newcommand{\Id}{\mathrm{Id}}
\newcommand{\D}{\mathrm{d}}
\newcommand{\X}{\mathfrak{X}}
\newcommand{\sx}{\mathop{\mathfrak{S}}\limits_{x,y,z}}
\newcommand{\s}{\mathfrak{S}}
\newcommand{\norm}[1]{\left\Vert#1\right\Vert ^2}
\newcommand{\M}{(M,\allowbreak\f,\allowbreak\xi,\allowbreak\eta,g)}
\newcommand{\Lf}{(G,\allowbreak\f,\allowbreak\xi,\allowbreak\eta,g)}
\newcommand{\thmref}[1]{The\-o\-rem~\ref{#1}}
\newcommand{\propref}[1]{Pro\-po\-si\-ti\-on~\ref{#1}}
\newcommand{\lemref}[1]{Lem\-ma~\ref{#1}}
\newcommand{\nT}{\norm{T}}
\newcommand{\nf}{\norm{\n \f}}
\newcommand{\corref}[1]{Corollary~\ref{#1}}
\newcommand{\dfnref}[1]{De\-fi\-ni\-ti\-on~\ref{#1}}
\newtheorem{thm}{Theorem}
\newtheorem{lem}[thm]{Lemma}
\newtheorem{prop}[thm]{Proposition}
\newtheorem{cor}[thm]{Corollary}
\newtheorem{defn}{Definition}

\begin{abstract}
A natural connection with totally skew-symmetric torsion on almost
contact manifolds with B-metric is constructed. The class of these
manifolds, where the considered connection exists,  is determined.
Some curvature properties for this connection, when the
corresponding curvature tensor has the properties of the curvature
tensor for the Levi-Civita connection and the torsion tensor is
parallel, are obtained. %
\end{abstract}



\maketitle

\section{Introduction}\label{intro}

The natural connections with totally skew-symmetric torsion (also
known as \emph{K\"ahler with torsion} (shortly, \emph{KT}-)
\emph{connections} or \emph{Bismut connections}) are of particular
interest in the string theory \cite{Stro} and in the Hermitian
geometry \cite{Bis}. The KT-geometry is a natural generalization
of the K\"ahler geometry, since when the torsion is zero the
KT-connection coincides with the Levi-Civita connection.
%
The KT-structures with a closed torsion 3-form (i.e. the so-called
\emph{strong} KT-structures) on a Hermitian manifold  have been
recently studied by many authors and they have also applications
in type II string theory and in 2-dimensional supersymmetric
$\sigma$-models \cite{GHR,IvPapa,Stro}.

The connections with totally skew-symmetric torsion arise in a
natural way in theoretical and mathematical physics.  According to
\cite{Gau, Fri-Iv2}, there exists a unique KT-connection on any
Hermitian manifold. In \cite{Fri-Iv2} and \cite{Fri-Iv} all almost
contact metric, almost Hermitian and $G_2$-structures admitting a
connection with totally skew-symmetric torsion tensor are
described. On almost complex manifolds with Nor\-den metric such a
connection is introduced and investigated in \cite{Mek2, Mek6,
Mek7}.

A quaternionic analog of K\"ahler geometry is given by the
\emph{hyper-K\"ahler with torsion} (shortly, \emph{HKT}-)
geometry. An HKT-manifold is a hyper-Hermitian manifold admitting
the so-called \emph{HKT-connection} -- a hyper-Hermitian
connection with totally skew-symmetric torsion, i.e. for which the
three KT-connections associated to the three Hermitian structures
coincide. This geometry is introduced in \cite{HoPa} and later
studied for instance in \cite{GranPoon, FiGran, BarDotVerb, BarFi,
Sw}. As an analogue of the known Hermitian structure on an almost
hypercomplex manifold it is introduced and investigated an
\emph{almost hypercomplex manifold with Hermitian and
anti-Hermitian metrics} in
\cite{GriManDim12,GriMan24,Man28,ManSek18}. On these manifolds
with a metric structure, generated by a pseudo-Riemannian metric
of neutral signature, a connection of the type of the
HKT-connection is defined and investigated in \cite{ManGri32}.

The goal of the present work is a consideration of a similar
problem of existence of connections of KT-type on almost contact
manifolds with B-metric. Our hope is that the wide application of
these connections on the several manifolds mentioned above will be
extended to manifolds in this case.

In this paper we consider a natural connection $D$ (i.e.
preserving the structure) with totally skew-symmetric torsion
tensor on almost contact manifolds with B-metric. These manifolds
are the odd-dimensional extension of the almost complex manifolds
with Norden metric  and the case with indefinite metrics
corresponding to almost contact metric manifolds.

In Section~\ref{sec:1} we give some necessary facts about the
considered manifolds.

In Section~\ref{sec:2} we define a natural connection $D$ with
totaly skew-symmet\-ric torsion on an almost contact manifold with
B-metric and call it shortly a $\f$KT-connection. We determine the
class of considered manifolds where $D$ exists.

In Section~\ref{sec:2a} we characterize the corresponding
curvature tensor $K$ of $D$ and related curvature properties in
the cases when $K$ has the properties of the curvature tensor $R$
for the Levi-Civita connection $\n$ on a $\F_0$-manifold (i.e. a
manifold with a $\n$-parallel almost contact B-metric structure).
We examine also the case of $D$-parallel torsion of $D$.
In Subsections~\ref{sec:3} and \ref{sec:4} we consider separately
the horizontal and vertical components of this class and
specialize the corresponding curvature properties given in the
former part of this section.


In Section~\ref{sec:5} we construct and characterize a 
family of 5-dimensional Lie groups as $\F_7$-manifolds with a
$D$-parallel torsion of the $\f$KT-connec\-ti\-on $D$.

\section{Almost contact manifolds with B-metric}\label{sec:1}

Let $(M,\f,\xi,\eta,g)$ be an almost contact manifold with
B-metric or an \emph{almost contact B-metric manifold}, i.e. $M$
is a $(2n+1)$-dimensional differentiable manifold with an almost
contact structure $(\f,\xi,\eta)$ consisting of an endomorphism
$\f$ of the tangent bundle, a vector field $\xi$, its dual 1-form
$\eta$ as well as $M$ is equipped with a pseudo-Riemannian metric
$g$  of signature $(n,n+1)$, such that the following algebraic
relations are satisfied
\begin{equation}\label{str}
\begin{array}{c}
\f\xi = 0,\quad \f^2 = -\Id + \eta \otimes \xi,\quad
\eta\circ\f=0,\quad \eta(\xi)=1,\\%
g(\f x, \f y ) = - g(x, y ) + \eta(x)\eta(y)
\end{array}
\end{equation}
for arbitrary $x$, $y$ of the algebra $\X(M)$ on the smooth vector
fields on $M$.

Further, $x$, $y$, $z$, $w$ will stand for arbitrary elements of
$\X(M)$.

The associated metric $\tilde{g}$ of $g$ on $M$ is defined by
$\tilde{g}(x,y)=g(x,\f y)+\eta(x)\eta(y)$. Both metrics are
necessarily of signature $(n,n+1)$. The manifold
$(M,\f,\xi,\eta,\tilde{g})$ is also an almost contact B-metric
manifold.

The structural group of almost contact manifolds with B-metric is
$\allowbreak\bigl(GL(n,\mathbb{C})\cap
O(n,n)\bigr)\allowbreak\times I_1$, i.e. it consists of real
square matrices of order $2n+1$ of the following type
\[
\left(%
\begin{array}{r|c|c}
  A & B & \vartheta^T\\ \hline
  -B & A & \vartheta^T\\ \hline
  \vartheta & \vartheta & 1 \\
\end{array}%
\right),\qquad %
\begin{array}{l}
  AA^T-BB^T=I_n,\\%
  AB^T+BA^T=O_n,
\end{array}%
\quad A, B\in GL(n;\mathbb{R}),
\]
where $\vartheta$ and its transpose $\vartheta^T$ are the zero row
$n$-vector and the zero column $n$-vector; $I_n$ and $O_n$ are the
unit matrix and the zero matrix of size $n$, respectively.

A classification of the almost contact manifolds with B-metric is
given in \cite{GaMiGr}. This classification is made with respect
to the tensor field $F$ of type (0,3) defined by
\begin{equation}\label{1.2}
F(x,y,z)=g\bigl( \left( \nabla_x \f \right)y,z\bigr),
\end{equation}
where $\nabla$ is the Levi-Civita connection of $g$. The tensor
$F$ has the following properties
\begin{equation}\label{1.3}
F(x,y,z)=F(x,z,y)=F(x,\f y,\f
z)+\eta(y)F(x,\xi,z)+\eta(z)F(x,y,\xi).
\end{equation}

This classification includes eleven basic classes $\F_1$, $\F_2$,
$\dots$, $\F_{11}$. In this work we pay attention to $\F_3$ and
$\F_7$. These basic classes are characterized by the conditions
\begin{gather}
\F_3:\quad \sx F(x,y,z)=0,\quad
F(\xi,y,z)=F(x,y,\xi)=0;\label{1.4}
\\%
\F_7:\quad \sx F(x,y,z)=0,\quad F(x,y,z)=-F(\f x,\f y,z)-F(\f
x,y,\f z),\label{1.4a}
\end{gather}
where $\s$ denotes the cyclic sum over three arguments.

The special class $\F_0$
, belonging to any other class $\F_i$ $(i=1,2,\dots,11)$, is
determined by the condition $F(x,y,z)=0$. Hence $\F_0$ is the
class of almost contact B-metric manifolds with $\n$-parallel
structures, i.e. $\n\f=\n\xi=\n\eta=\n g=\n \tilde{g}=0$.

The components of the inverse matrix of $g$ are denoted by
$g^{ij}$ with respect to the basis $\{e_i\}$ of the tangent space
$T_pM$ of $M$ at an arbitrary point $p\in M$.

By analogy with the square norm of $\nabla J$ for an almost
complex structure $J$, we define the \emph{square norm of $\nabla
\f$} by
\begin{equation}\label{snf}
    \norm{\nabla \f}=g^{ij}g^{ks}
    g\bigl(\left(\nabla_{e_i} \f\right)e_k,\left(\nabla_{e_j}
    \f\right)e_s\bigr).
\end{equation}

It is clear, the equality $\norm{\nabla \f}=0$ is valid if $\M$ is
a $\F_0$-manifold, but the inverse implication is not always true.
An almost contact B-metric manifold having a zero square norm of
$\n\f$ we call an \emph{isotropic-$\F_0$-manifold}.

As it is known \cite{KoNo}, an arbitrary  vector field $\xi$ is
\emph{Killing} if the Lie differential of the metric $g$ with
respect to $\xi$ is zero. Then, the vector field $\xi$ is Killing
on $(M,\f,\xi,\eta,g)$ if and only if
\begin{equation}\label{kil}
\left(\n_x \eta\right)y+\left(\n_y \eta\right)x=0.
\end{equation}

\begin{lem}\label{37}
The class $\F_3\oplus\F_7$ is the class of almost contact B-metric
manifolds $(M,\f,\xi,\eta,\allowbreak g)$ with a Killing vector
field $\xi$ and a zero cyclic sum $\s$ of $F$.
\end{lem}
\begin{proof}
Equation \eqref{kil} and direct calculations yield the statement.
\end{proof}

The Nijenhuis tensor $N$ of the structure is defined by $N := [\f,
\f]+ \D\eta\otimes\xi$, where $[\f, \f](x, y )=\left[\f x,\f y
\right]+\f^2\left[x,y \right]-\f\left[\f x,y \right]-\f\left[x,\f
y \right]$. Hence, $N$ in terms of the covariant derivatives has
the form
\begin{equation}\label{N}
\begin{split}
N(x,y)=&\left(\n_{\f x}\f\right)y-\left(\n_{\f
y}\f\right)x-\f\left(\n_{x}\f\right)y
+\f\left(\n_{y}\f\right)x\\%
&+\left(\n_{x}\eta\right)y.\xi-\left(\n_{y}\eta\right)x.\xi.
\end{split}
\end{equation}
The Nijenhuis (0,3)-tensor is determined by
$N(x,y,z)=g\left(N(x,y),z\right)$. Then from \eqref{N} and
\eqref{1.2} we have
\begin{equation}\label{NF}
\begin{split}
N(x,y,z)=&F(\f x,y,z)-F(\f y,x,z)-F(x,y,\f z)+F(y,x,\f
z)\\%
&+F(x,\f y,\xi)\eta(z)-F(y,\f x,\xi)\eta(z).
\end{split}
\end{equation}

In \cite{Man} there are considered the linear projectors $h$ and
$v$ over $T_pM$ which split (orthogonally and invariantly with
respect to the structural group) any vector $x$ into a horizontal
component $h(x)=-\f^2x$ and a vertical component
$v(x)=\eta(x)\xi$.
The decomposition $T_pM=h(T_pM)\oplus v(T_pM)$ generates the
corresponding distribution of the basic tensors $F$, which gives
the horizontal component $\F_3$ and the vertical component $\F_7$
of the class $\F_3\oplus\F_7$.

Moreover, for the considered classes we have the following
\begin{lem}
The Nijenhuis tensor $N$ on an almost contact B-metric manifold
$\M$ has the corresponding form:
\begin{enumerate}[(i)]
    \item If $\M\in\F_3\oplus\F_7$ then
\[
\begin{array}{l}
    N(x,y)=2\left(\n_{\f x} \f\right)y-2\f\left(\n_{x}
    \f\right)y+ 2\left(\n_{x} \eta\right)y.\xi,\\%
    h(N(x,y))=-2\f^2\left(\n_{\f x} \f\right)y-2\f\left(\n_{x}
    \f\right)y,\\%
    v(N(x,y))=4\left(\n_{x} \eta\right)y.\xi;
\end{array}
\]
    \item If $\M\in\F_3$, then
\[
N(x,y)=h(N(x,y))=2\left(\n_{\f x} \f\right)y-2\f\left(\n_{x}
    \f\right)y;
\]
    \item If $\M\in\F_7$, then
\[
N(x,y)=v(N(x,y))=4\left(\n_{x} \eta\right)y.\xi.
\]
\end{enumerate}
\end{lem}

\begin{proof}
The expression of $N$ in (i) follows from \eqref{NF} by virtue of
equalities \eqref{1.2}--\eqref{kil}. Hence,
$v(N(x,y))=N(x,y,\xi)\xi=2F(\f x,y,\xi)\xi+2\left(\n_{x}
\eta\right)y.\xi=4\left(\n_{x} \eta\right)y.\xi$, using that $\xi$
is a Kill\-ing vector field. After that, since
$h(N(x,y))\allowbreak=N(x,y)-v(N(x,y))$, we obtain the
corresponding expression of $h(N(x,y))$.

In case (ii), according to \eqref{1.4}, we have $F(x,y,\xi)=0$ and
therefore $\left(\n_{x} \eta\right)y=0$. Then the form of $N$ in
(i) implies $N(x,y)=2\left(\n_{\f x} \f\right)y-2\f\left(\n_{x}
    \f\right)y$. In addition to that, since $v(N(x,y))=0$, then
    $h(N(x,y))=N(x,y)$ holds.

In case (iii), we have $F(x,y,z)=-F(\f x,\f y,z)-F(\f x,y,\f z)$
from \eqref{1.4a}. According to properties \eqref{1.3}, we get
$F(x,y,z)=\eta(y)F(x,\xi,z)+\eta(z)F(x,y,\xi)$ and by \eqref{1.2}
the relation  $\left(\n_{x}
\f\right)y=-\eta(y)\f\n_x\xi-\left(\n_{x} \eta\right)\f y.\xi$
holds. Applying the latter equality in the expression of $N$ in
(i), we obtain that $N(x,y)=4\left(\n_{x} \eta\right)y.\xi$. Then
the
equality $N(x,y)=v(N(x,y))$ follows immediately.%
\end{proof}

Let $R=\left[\n,\n\right]-\n_{[\ ,\ ]}$ be the curvature tensor of
type (1,3) of $\nabla$. We denote the curvature tensor of type
$(0,4)$  $R(x,y,z,w)$ $=g(R(x,y)z,w)$ by the same letter.

The Ricci tensor $\rho$ for the curvature tensor $R$ and the
scalar curvature $\tau$ for $R$ are defined by the traces
$\rho(y,z)=g^{ij}R(e_i,y,z,e_j)$ and $\tau=g^{ij}\rho(e_i,e_j)$,
respectively.

Similarly, the Ricci tensor and the scalar curvature are
determined for each \emph{curvatu\-re-like tensor} $L$, i.e.
 for the (0,4)-tensor $L$ with the following properties:
\begin{gather}
    L(x,y,z,w)=-L(y,x,z,w)=-L(x,y,w,z),\label{1.10}
    \\%
    \sx L(x,y,z,w)=0 \qquad \text{(the first Bianchi
    identity)}.\label{1.10B}
\end{gather}

A curvature-like tensor we call a \emph{$\f$-K\"ahler-type tensor}
on $\M$ if the following identity is valid
\begin{equation}\label{1.12}
    L(x,y,\f z,\f w)=-L(x,y,z,w).
\end{equation}
This property is characteristic of the curvature tensor $R$ on a
$\F_0$-manifold. Moreover, \eqref{1.12} is similar to the
corresponding property for a K\"ahler-type tensor with respect to
$J$ on an almost complex manifold with Norden metric.


\section{$\f$KT-connection}\label{sec:2}


\begin{defn}
A linear connection $D$ is called a \emph{natural connection} on
the manifold  $(M,\f,\allowbreak\xi,\eta,g)$ if the almost contact
structure $(\f,\xi,\eta)$ and the B-metric $g$ are parallel with
respect to $D$, i.e. $D\f=D\xi=D\eta=Dg=0$.
\end{defn}

If $T$ is the torsion tensor of $D$, i.e. $T(x,y)=D_x y-D_y x-[x,
y]$, then the corresponding tensor field of type (0,3) is
determined by $T(x,y,z)=g(T(x,y),z)$.

Let $\n$ be the Levi-Civita connection generated by $g$. Then we
denote
\begin{equation}\label{1}
D_xy=\n_xy+Q(x,y).
\end{equation}
Furthermore, we use the denotation
$Q(x,y,z)=g\left(Q(x,y),z\right)$.
\begin{thm}\label{thm-KT}
The linear connection $D$ is a natural connection on the manifold
$(M,\f,\xi,\allowbreak\eta,g)$ if and only if %
\begin{gather}
 Q(x,y,\f z)-Q(x,\f y,z)=F(x,y,z),\label{1a}
 \\%
 Q(x,y,z)=-Q(x,z,y).\label{1b}
\end{gather}
\end{thm}
\begin{proof}
From \eqref{1} we have
\begin{gather}
g(D_x \f y,z)=g(\n_x \f y,z)+ Q(x,\f y,z),  \nonumber
\\%
g(D_x y,\f z)=g(\n_x y,\f z)+ Q(x,y,\f z).\nonumber
\end{gather}
Subtracting the last two equations we immediately obtain
$$
g\bigl(\left(D_x \f \right) y, z\bigr)=F(x, y, z)+ Q(x,\f y,z)-
Q(x,y,\f z).
$$
Then $D\f=0$ is equivalent to \eqref{1a}.

After that we get consecutively
$$
\left(D_x g \right) (y, z)=x g(y, z)- g(D_x y,z)- g(y,D_x z)
=-Q(x,y,z)-Q(x,z,y).
$$
Therefore, $Dg=0$ is valid if and only if \eqref{1b} holds.

Equality \eqref{1} implies $g(D_x\xi,z)=g(\n_x\xi,z)+Q(x,\xi,z)$.
Because of the identity
\[
g(\n_x\xi,z)\allowbreak=F(x,\xi,\f z),
\]
we have that $D\xi=0$ is equivalent to the relation $F(x,\xi,\f
z)+Q(x,\xi,z)=0$. The last equality is a consequence of
\eqref{1a}.

Since $\eta(\cdot)=g(\cdot,\xi)$, then supposing $Dg=0$ we have
$D\xi=0$ if and only if $D\eta=0$.
\end{proof}

\begin{defn}
A natural connection $D$ is called a \emph{$\f$KT-connection} on
the manifold $(M,\f,\xi,\allowbreak\eta,g)$ if the torsion tensor
$T$ of $D$ is totally skew-symmetric, i.e. a 3-form.
\end{defn}

For a $\f$KT-connection $D$ with torsion tensor $T$ and tensor of
deformation $Q$ from $\n$ to $D$ determined by \eqref{1} we have
\begin{equation}\label{TQ}
T(x,y,z)=2Q(x,y,z).
\end{equation}
Therefore, $Q$ is also a 3-form.

\begin{lem}\label{lem-NT}
Let $D$ be a $\f$KT-connection on $(M,\f,\xi,\eta,g)$. Then the
torsion $T$ of $D$ and the Nijenhuis tensor $N$ on $\M$ have the
following relation
\[
N(x,y,z)=T(x,y,z)+T(x,\f y,\f z)+T(\f x,y,\f z)-T(\f x,\f y,z). %
\]
\end{lem}
\begin{proof}
The statement follows from \eqref{NF}, \thmref{thm-KT} and
\eqref{TQ}.
\end{proof}

\begin{prop}
\label{prop-exD-37} If a $\f$KT-connection exists on $\M$ then
$\xi$ is a Killing vector field and $\s F=0$, i.e.
$(M,\f,\xi,\eta,g)$ belongs to the class $\F_3\oplus\F_7$.
\end{prop}
\begin{proof}
Let us assume that a $\f$KT-connection $D$ exists. Then, according
to \thmref{thm-KT}, the conditions \eqref{1a} and \eqref{1b} are
valid for $Q=\frac{1}{2}T$. Taking the cyclic sum in \eqref{1a},
because of \eqref{TQ}, we obtain
\[
\begin{split}
&F(x,y,z)+F(y,z,x)+F(z,x,y)\\%
&=\frac{1}{2}\left\{T(x,y,\f z)+T(y,z,\f x)+T(z,x,\f y)\right.\\%
&\phantom{=\frac{1}{2}}\left.-T(x,\f y,z)-T(y,\f z,x) -T(z,\f
x,y)\right\}.
\end{split}
\]
Since $T$ is totally skew-symmetric therefore $\sx F(x,y,z)=0$.

Using \eqref{1a} and \eqref{TQ} we obtain $2F(x,\f
y,\xi)=T(x,y,\xi)$ and since $\left(\n_x
\eta\right)y\allowbreak=F(x,\f
y,\xi)$ we have %
\begin{equation}\label{n-eta-T}
2\left(\n_x \eta\right)y=T(x,y,\xi).
\end{equation}
Then, bearing in mind \eqref{kil}, the statement $T$ is a 3-form
implies that $\xi$ is Killing and according to \lemref{37}, the
statement
is true. %
\end{proof}

\begin{thm}
\label{prop-37} Let $(M,\f,\xi,\eta,g)$ be in the class
$\F_3\oplus\F_7$. Then there exists a  $\f$KT-connect\-ion $D$
determined by
    \[
        g(D_xy,z)=g(\n_xy,z)+\frac{1}{2}T(x,y,z),
    \]
    where the torsion $T$ is defined by
\begin{equation}\label{T37wedge} %
T(x,y,z)=\left(\eta\wedge \D\eta\right)(x,y,z)+\frac{1}{4}\sx
N(x,y,z)%
\end{equation} %
or equivalently
\begin{equation}\label{T37} %
T(x,y,z)=-\frac{1}{2} \sx\bigl\{F(x,y,\f
z)-3\eta(x)F(y,\f z,\xi)\bigr\}. %
\end{equation} %
\end{thm}
\begin{proof}
Let us suppose that $\M\in\F_3\oplus\F_7$ (i.e. $\xi$ is a Killing
vector field and $\s F=0$) and define a connection $D$ with
torsion $T$ determined by \eqref{T37}.

We substitute $\f z$ for $z$ in \eqref{T37}, apply properties
\eqref{1.3}--\eqref{1.4a} and obtain consequently
\begin{equation}\label{T1}
\begin{split}
T(x,y,\f z)=-\frac{1}{2} \bigl\{&F(x,y,\f^2 z)+3\eta(x)F(y,z,\xi)\\
&+F(y,\f z, \f x)+3\eta(y)F(z,x,\xi)+F(\f z,x,\f y)\bigr\}\\
=-\frac{1}{2} \bigl\{&-F(x,y,z)+2\eta(x)F(y,z,\xi)+2\eta(z)F(x,y,\xi)\\
&+F(y,z,x)+3\eta(y)F(z,x,\xi)+F(\f z,x,\f y)\bigr\}\\
\end{split}
\end{equation}
Analogously,  from \eqref{T37} we obtain the following
\begin{equation}\label{T2}
\begin{split}
T(x,\f y,z)=-\frac{1}{2} \bigl\{&F(x,y,z)+3\eta(x)F(y,z,\xi)+2\eta(y)F(z,x,\xi)\\
&+F(\f y,z, \f x)-F(z,x,y)+2\eta(z)F(x,y,\xi)\bigr\}\\
\end{split}
\end{equation}
Then we subtract \eqref{T2} from \eqref{T1} and applying
\eqref{1.3}--\eqref{1.4a} again, we get
\begin{equation}\label{TTF}
T(x,y,\f z)-T(x,\f y,z)=2F(x,y,z).
\end{equation}
Therefore \eqref{1a} is valid for the connection $D$ with property
$T=2Q$.

Moreover, we form the expression $T(x,y,z)+T(x,z,y)$ using
\eqref{T37}. By virtue of the consequence $F(y,\f z,\xi)=-F(z,\f
y,\xi)$ of \eqref{kil} and the symmetry of $F$ by the latter two
arguments known from \eqref{1.3}, we obtain
\[
T(x,y,z)+T(x,z,y)=\frac{1}{2}\sx F(\f x, y,z).
\]
According to \eqref{1.3} and because $\xi$ is Killing, we have
\[
T(x,y,z)+T(x,z,y)=\frac{1}{2}\sx F(\f x, \f y,\f z).
\]
Since the cyclic sum of $F$ over three arguments is zero for
$\F_3\oplus\F_7$ then \eqref{1b} is valid and $T$ is a 3-form.
Moreover, according to \thmref{thm-KT}, the connection $D$ with
torsion $T$ determined by \eqref{T37} is natural for the almost
contact B-metric
structure. Therefore $D$ is a $\f$KT-connection. %

In addition to that, according to \eqref{n-eta-T} and \eqref{kil},
for a $\f$KT-connection we have
\begin{equation}\label{detaT}
\D\eta(x,y)=2\left(\n_x\eta\right)y=T(x,y,\xi)=2F(x,\f y,\xi).
\end{equation}
Hence, the following equalities hold
\begin{equation}\label{eta-wedge}
\begin{split}
(\eta\wedge \D\eta)(x,y,z)&=\sx \eta(x)\D\eta(y,z)=\sx
\eta(x)T(x,y,\xi)\\%
&=\sx \eta(x)F(x,\f y,\xi).
\end{split}
\end{equation}
By virtue of the formula in \lemref{lem-NT}, we obtain
\begin{equation}\label{sNT} %
\begin{split}
\sx N(x,y,z)=3T(x,y,z)&+T(x,\f y,\f z)\\%
&+T(\f x,y,\f z)+T(\f x,\f y,z). %
\end{split}
\end{equation} %
Using \eqref{TTF}, \eqref{eta-wedge}, \eqref{sNT}  and the totally
skew-symmetry of $T$, we get the equivalence of \eqref{T37wedge}
and \eqref{T37}.
\end{proof}

Bearing in mind \eqref{1.2} and $F(x,\f y,\xi)=\left(\n_x
\eta\right)y=g\left(\n_x \xi,y\right)$, the torsion (0,3)-field
$T$ of $D$, given in \eqref{T37}, has the following form of type
(1,2)
\begin{equation}\label{T37a}%
\begin{split}
T(x,y)=\frac{1}{2}\bigl\{&2\left(\n_x\f\right)\f
y-\left(\n_y\f\right)\f x+\left(\n_{\f y}\f\right)x\\%
&+3
\eta(x)\n_y\xi-4\eta(y)\n_x\xi+2\left(\n_x\eta\right)y.\xi\bigr\}.
\end{split}
\end{equation} %

Formula \eqref{T37a} implies the following%
\begin{cor}\label{cor-T-KT}
The $\f$KT-connection $D$ determined by \eqref{T37a} has the
properties:
\[
T(x,\f y)=\f T(x,y)-2\left(\n_x\f\right)y,\qquad T(\f x,y)=\f
T(x,y)+2\left(\n_y\f\right)x.
\]%
\end{cor}
\begin{proof}
The former equality is a consequence of \eqref{TTF}. The latter
one follows from the former equality and the totally skew-symmetry
of $T$.
\end{proof}

\section{Curvature properties of the $\f$KT-connection}\label{sec:2a}

The curvature tensor $K$ for a linear connection $D$ is determined
as usual by $K(x,y)z=\left[D_x,D_y\right]z-D_{[x,y]}z$ and
$K(x,y,z,w)=g\left(K(x,y)z,w\right)$.

According to \eqref{1} and \eqref{1b}, the curvature tensor $K$ of
a
connection $D$ with the condition $Dg=0$ has the form%
\begin{equation}\label{KRQ}
\begin{split}
    K(x,y,z,w)=&R(x,y,z,w)+\left(D_x Q\right)(y,z,w)-\left(D_y
    Q\right)(x,z,w)\\%
    &+g\left(Q(y,z),Q(x,w)\right)-g\left(Q(x,z),Q(y,w)\right)\\%
    &+Q\left(T(x,y),z,w\right).
\end{split}
\end{equation}%
Using \eqref{KRQ} and \eqref{TQ}, we obtain the following known
equality for a connection $D$ with totally skew-symmetric torsion
$T$ and a $D$-parallel metric $g$ (cf. \cite{Fri-Iv})
\begin{equation}\label{KRT}
\begin{split}
    K(x,y,z,w)=&R(x,y,z,w)+\frac{1}{2}\left(D_x T\right)(y,z,w)
                        -\frac{1}{2}\left(D_y T\right)(x,z,w)
  \\%
        &+\frac{1}{4}g\left(T(x,y),T(z,w)\right)
        +\frac{1}{4}\sx
        \left\{g\left(T(x,y),T(z,w)\right)\right\}.
\end{split}
\end{equation}

Since the scalar curvatures of $D$ and $\n$ are determined by the
equalities $\tau^D=g^{ij}g^{ks}K(e_i,e_k,e_s,e_j)$ and
$\tau=g^{ij}g^{ks}R(e_i,e_k,e_s,e_j)$, respectively, then relation
\eqref{KRT} implies
\begin{equation}\label{tDt}
\tau^D=\tau-\frac{1}{4}\nT,
\end{equation}
where $\nT=g^{ij}g^{ks}g\bigl(T(e_i,e_k),T(e_j,e_s)\bigr)$.

\begin{thm}\label{thm-fiKelerov}
The curvature tensor $K$ of a $\f$KT-connection $D$ on $\M$
 is of $\f$-K\"ahler type if and only if it has the form
\begin{equation}\label{KRsT}
\begin{split}
    K(x,y,z,w)=R(x,y,z,w)&+\frac{1}{4}g\left(T(x,y),T(z,w)\right)\\%
        &-\frac{1}{12}\sx \left\{g\left(T(x,y),T(z,w)\right)\right\}.
\end{split}
\end{equation}
\end{thm}
\begin{proof}
Since $K$ is a curvature tensor of a natural connection then
property \eqref{1.12} is valid for $K$. Hence, $K$ becomes a
$\f$-K\"ahler-type tensor when $\s K=0$ is satisfied additionally.
The first Bianchi identity for a $\f$KT-connection $D$ with
torsion $T$ can be written in the form
\[
\sx K(x,y,z,w)=\sx \left(D_xT\right)(y,z,w)+\sx
\left\{g\left(T(x,y),T(z,w)\right)\right\},
\]
because of $Dg=0$ and $T$ is a 3-form.
Then, $K$ is of $\f$-K\"ahler type if and only if the following is
valid
\begin{equation}\label{DTgT}
\sx \left(D_x T\right)(y,z,w)=-\sx
\left\{g\left(T(x,y),T(z,w)\right)\right\}.
\end{equation}
In this case, according to \eqref{KRT} and \eqref{DTgT}, $K$ has
the form
\begin{equation}\label{KRDT}
\begin{split}
    K(x,y,z,w)=&R(x,y,z,w)-\frac{1}{2}\left(D_z T\right)(x,y,w)
                          \\%
        &-\frac{1}{4}g\left(T(y,z),T(x,w)\right)
        -\frac{1}{4}g\left(T(z,x),T(y,w)\right).
\end{split}
\end{equation}
Hence, we change the positions of $y$ and $w$ in \eqref{KRDT} and
add the obtained equality to the original form of \eqref{KRDT}. In
the result we exchange $z$ and $w$ and subtract the obtained
equality from the original one. In such a way we get \eqref{KRsT}.
\end{proof}

Let us consider the case when  $D$ has a $D$-parallel torsion
tensor $T$. Then, using \eqref{KRT}, the curvature tensor $K$ gets
the following form
\begin{equation}\label{KR-DT=0}
\begin{split}
    K(x,y,z,w)=R(x,y,z,w)
        &+\frac{1}{4}g\left(T(x,y),T(z,w)\right)\\%
        &+\frac{1}{4}\sx \left\{g\left(T(x,y),T(z,w)\right)\right\}.
\end{split}
\end{equation}%

\begin{thm}\label{thm-K-sK=DT=0}
Let a $\f$KT-connection $D$ have a $D$-parallel torsion tensor $T$
on $\M$. Then the following conditions are equivalent:
\begin{enumerate}[(i)]
    \item The curvature tensor $K$ is a $\f$-K\"ahler-type tensor and it has the form
\begin{equation}\label{KR3-sK=DT=0}
    K(x,y,z,w)=R(x,y,z,w)
        +\frac{1}{4}g\left(T(x,y),T(z,w)\right).
\end{equation}%
    \item The torsion 3-form $T$ is closed;
    \item The following equality is valid
\begin{equation}\label{sgTT=0}
    \sx \left\{g\left(T(x,y),T(z,w)\right)\right\}=0.
\end{equation}
\end{enumerate}
\end{thm}
\begin{proof}
Bearing in mind \eqref{KR-DT=0} and \eqref{KRsT}, we get that (i)
and (iii) are equivalent, according to \thmref{thm-fiKelerov}.

Since the torsion tensor $T$ of a $\f$KT-connection $D$ is a
3-form, then the exterior derivative of $T$ has the expression
\[
\begin{split}
    \D{T}(x,y,z,w)=&\sx\left(D_x T\right)(y,z,w)
                        -\left(D_w T\right)(x,y,z)
  \\%
        &+2\sx
        \left\{g\left(T(x,y),T(z,w)\right)\right\}.
\end{split}
\]
\[
    \D{T}(x,y,z,w)=\sx\left(D_x T\right)(y,z,w)
                        -\left(D_w T\right)(x,y,z)+2\sx
        \left\{g\left(T(x,y),T(z,w)\right)\right\}.
\]
If $DT=0$ then $\D T(x,y,z,w)=
        2\sx
        \left\{g\left(T(x,y),T(z,w)\right)\right\}$ is valid.
        Therefore, conditions (ii) and (iii) are equivalent.
\end{proof}

\begin{prop}\label{prop-Rf37}
Let $\M$ be a manifold in $\F_3\oplus\F_7$ with a $D$-parallel
torsion tensor $T$ and a curvature tensor $K$ of $\f$-K\"ahler
type for the $\f$KT-connection $D$. The curvature tensor $R$ has
the following properties with respect to the structure:
 \[
\begin{split}
R(x,y,\f z,\f w) =&- R(x, y, z, w)-\frac{1}{4}g\bigl(T(x,y),
T(z,w)+T(\f z,\f w)\bigr),\\%
R(x,y,z,\xi) =&\frac{1}{2}g\bigl(T(x,y), \n_z \xi\bigr), %
\end{split}
\]
where $T$ is determined by \eqref{T37a}.
\end{prop}
\begin{proof}
Bearing in mind \thmref{thm-K-sK=DT=0}, we apply property
\eqref{1.12} for $K$ determined by \eqref{KR3-sK=DT=0} and obtain
the former equality in the statement. It implies the latter
equality since $\f\xi=0$ and $T(z,\xi)=-2\n_z\xi$.
\end{proof}

Further, in the next two subsections, we consider curvature
properties of the $\f$KT-con\-nec\-tion $D$ in the two special
cases -- on the horizontal component $\F_3$ and on the vertical
com\-ponent $\F_7$ of the class $\F_3\oplus\F_7$, where $D$
exists.

\subsection{Curvature properties of the $\f$KT-connection on the
horizontal component $\F_3$}\label{sec:3}

Since the horizontal restriction of any $\F_3$-manifold is an
almost complex manifold with Norden metric in the class $\W_3$
(also known as a quasi-K\"ahler manifold with Norden metric), then
the curvature properties can be obtained in an analogous way as in
\cite{Mek2}. Besides, the curvature properties of a
$\F_3$-manifold are consequences of the corresponding properties
for any manifold in $\F_3\oplus\F_7$, given in the general part of
Section~\ref{sec:2a}.

Bearing in mind \eqref{T37a}, the torsion of
the $\f$KT-connection $D$ on a $\F_3$-manifold has the form %
\begin{equation}\label{T3}%
T(x,y)=\frac{1}{2}\bigl\{2\left(\n_x\f\right)\f
y-\left(\n_y\f\right)\f x+\left(\n_{\f y}\f\right)x\bigr\}.
\end{equation} %

According to \thmref{thm-fiKelerov}, the curvature tensor $K$ of
this connection on a $\F_3$-manifold $\M$
 is of $\f$-K\"ahler type if and only if $K$ has the form in
 \eqref{KRsT}, where $T$ is determined by \eqref{T3}.
Hence, by virtue of \eqref{tDt}, \eqref{T3} and \eqref{snf}, we
get immediately the following relation between the scalar
curvatures for the $\f$KT-connection $D$ and the Levi-Civita
connection $\n$:
\[
\tau^D =\tau + \frac{3}{8}\nf.
\]

The following statement is a direct consequence of the latter
equality.
\begin{cor}
Let $\M$ be a $\F_3$-manifold with a curvature tensor of
$\f$-K\"ahler type for the $\f$KT-connection $D$. Then $\M$ is an
isotropic-$\F_0$-manifold if and only if the scalar curvatures for
$D$ and $\n$ are equal.
\end{cor}

According to \propref{prop-Rf37} and \corref{cor-T-KT}, we obtain
the following
\begin{cor}
Let $\M$ be a $\F_3$-manifold with a $D$-parallel torsion tensor
$T$ and a curvature tensor $K$ of $\f$-K\"ahler type for the
$\f$KT-connection $D$. The curvature tensor $R$ has the following
properties with respect to the structure:
 \[
\begin{split}
R(x,y,z,\xi) &= 0,
\\
R(x,y,\f z,\f w) &=- R(x, y, z, w) %
+\frac{1}{2}g\Bigl(%
T(x,y), \left(\n_{\f z} \f\right)w+\left(\n_{w} \f\right)\f
z\Bigr),
\end{split}
\]
where $T$ is determined by \eqref{T3}.
\end{cor}

\subsection{Curvature properties of the $\f$KT-connection on the
vertical component $\F_7$}\label{sec:4}

Now, we will pay attention on the other case when the manifold
$\M$ belongs to the class $\F_7$, i.e. the manifold has the
properties: a zero cyclic sum of $F$, a Killing $\xi$ and a
nonzero $\n\xi$ (or the horizontal component of $F$ is zero). For
such a manifold the torsion of
the $\f$KT-connection $D$ has the form %
\begin{equation}\label{T7}%
T(x,y)=2\left\{\eta(x)\n_y\xi-\eta(y)\n_x\xi+\left(\n_x\eta\right)y.\xi\right\}
\end{equation} or
\[
T(x,y,z)=\left(\eta\wedge \D\eta\right)(x,y,z)
=2\sx\bigl\{\eta(x)F(y,\f z,\xi)\bigr\}.
\]

Using \eqref{KRT} and \eqref{T7}, by direct calculations we get
the following form of $K$ on a $\F_7$-manifold:
\begin{equation*}\label{KR7-all}
\begin{split}
K(x,y,z,w) =& R(x, y, z, w) + 2 \left(\n_x \eta\right)y\left(\n_z
\eta\right)w\\%
&+\left(\n_y \eta\right)z\left(\n_x \eta\right)w+\left(\n_z
\eta\right)x\left(\n_y
\eta\right)w\\%
&-\eta(x)\bigl\{%
 \left(\n_y \n_z \eta\right)w%
-\left(\n_y \n_w\eta\right)z \bigr\}\\%
&+\eta(y)\bigl\{%
\left(\n_x \n_z \eta\right)w-\left(\n_x \n_w \eta\right)z
\bigr\}\\%
&-\eta(z)\bigl\{%
 \left(\n_x \n_y \eta\right)w-\left(\n_y \n_x \eta\right)w
+\left(\n_y \n_w \eta\right)x-\left(\n_x \n_w \eta\right)y\bigr\}\\%
&+\eta(w)\bigl\{%
 \left(\n_x \n_y \eta\right)z-\left(\n_y \n_x \eta\right)z%
+\left(\n_y \n_z \eta\right)x-\left(\n_x \n_z \eta\right)y\bigr\}\\%
&-\eta(y)\eta(z)g\left(\n_x \xi,\n_w \xi\right)%
+\eta(x)\eta(z)g\left(\n_y \xi,\n_w \xi\right)\\
&-\eta(x)\eta(w)g\left(\n_y \xi,\n_z \xi\right)%
 +\eta(y)\eta(w)g\left(\n_x \xi,\n_z \xi\right).
\end{split}
\end{equation*}

Since $\M\in\F_7$, because of%
\[
\left(\n_x\f\right)y=-g\left(\n_x \xi,\f y\right)
\xi-\eta(y)\f\n_x \xi,
\]
we obtain the formula for the square norm of $\n\f$ as follows %
\begin{equation}\label{snf7} %
\nf=-2g^{ij}g\left(\n_{e_i} \xi,\n_{e_j} \xi\right).
\end{equation} %

Equations 
\eqref{tDt}, \eqref{T7} and \eqref{snf7} yield the following
relation between the scalar curvatures for $D$ and $\n$:
\[
\tau^D=\tau+\frac{3}{2}\nf.
\]

The last equality implies the following
\begin{cor}
Let $\M$ belong to $\F_7$. Then $\M$ is an
isotropic-$\F_0$-mani\-fold if and only if the scalar curvatures
for $D$ and $\n$ are equal.
\end{cor}
\begin{prop}\label{prop-KR7}
The curvature tensor $K$ of the $\f$KT-connection $D$ on a
$\F_7$-ma\-ni\-fold $\M$
 is of $\f$-K\"ahler type if and only if it has the form
\begin{equation}\label{KR7}
\begin{split}
&K(x,y,z,w) = R(x, y, z, w)%
\\%
&\phantom{K(x,y,z,w) =}
-\eta(y)\eta(z)g\left(\n_x \xi,\n_w \xi\right)%
+\eta(x)\eta(z)g\left(\n_y \xi,\n_w \xi\right)\\
&\phantom{K(x,y,z,w) =}
-\eta(x)\eta(w)g\left(\n_y \xi,\n_z \xi\right)%
 +\eta(y)\eta(w)g\left(\n_x \xi,\n_z \xi\right)\\
&+ \frac{1}{3}\left\{2 \left(\n_x \eta\right)y\left(\n_z
\eta\right)w-\left(\n_y \eta\right)z\left(\n_x
\eta\right)w-\left(\n_z \eta\right)x\left(\n_y
\eta\right)w\right\}.
\end{split}
\end{equation}
\end{prop}
\begin{proof}
Bearing in mind \eqref{KRsT} and \eqref{T7}, we obtain
\eqref{KR7}.
\end{proof}

We deduce from \propref{prop-KR7}, using \eqref{snf7}, the
following relation between
the Ricci tensors for $D$ and $\n$ on $\M\in\F_7$ with curvature
tensor of $\f$-K\"ahler type
\[
\rho^D(y,z)=\rho(y,z)-2g\left(\n_y \xi,\n_z
\xi\right)+\frac{1}{2}\eta(y)\eta(z)\nf.
\]

\begin{prop}\label{prop-Rf7}
Let the $\f$KT-connection $D$ have a curvature tensor of
$\f$-K\"ahler type on $\M\in\F_7$. The curvature tensor $R$ of the
Levi-Civita connection $\n$ has the following properties with
respect to the structure:
\begin{equation}\label{R7f}
\begin{split}
R(x,y,\f z,\f w) =&- R(x, y, z, w)\\
&
+\eta(y)\eta(z)g\left(\n_x \xi,\n_w \xi\right)%
-\eta(x)\eta(z)g\left(\n_y \xi,\n_w \xi\right)\\
&
+\eta(x)\eta(w)g\left(\n_y \xi,\n_z \xi\right)%
-\eta(y)\eta(w)g\left(\n_x \xi,\n_z \xi\right)\\
&+\frac{1}{3}\bigl\{%
\left(\n_x \eta\right)z\left(\n_y \eta\right)w%
-\left(\n_x \eta\right)w\left(\n_y \eta\right)z \\%
&\phantom{+\frac{1}{3}\bigl\{} %
+\left(\n_x\eta\right)\f z\left(\n_y \eta\right)\f w -\left(\n_x
\eta\right)\f w\left(\n_y \eta\right)\f z\bigr\},
\end{split}
\end{equation} %
\begin{equation}
\label{R7xi} %
R(x,y,z,\xi) =
\eta(x)g\left(\n_y \xi,\n_z \xi\right)%
-\eta(y)g\left(\n_x \xi,\n_z \xi\right). %
\end{equation} %
\end{prop}
\begin{proof}
Since $D\f=0$ then $K$ satisfies \eqref{1.12}, i.e. $K$ is a
$\f$-K\"ahler-type tensor. In \eqref{KR7} we substitute $\f z$ for
$z$ and $\f w$ for $w$ and add the obtained equality to
\eqref{KR7}. Then we obtain \eqref{R7f}, which gives the defect of
property \eqref{1.12} for $R$. Property \eqref{R7xi} follows from
\eqref{KR7} using $K(x,y,z,\xi)=0$, because of $D\xi=0$.
\end{proof}


Now, let us consider the $\f$KT-connection $D$ with $D$-parallel
torsion $T$, i.e. $DT=0$, on almost contact B-metric manifolds
$\M$ in the class $\F_7$.

According to \thmref{thm-K-sK=DT=0}, \eqref{KR-DT=0} and
\eqref{T7}, the following two statements hold.

\begin{prop}\label{prop-K7-DT=0}
Let the $\f$KT-connection $D$ have a $D$-parallel torsion tensor
$T$ on a $\F_7$-ma\-ni\-fold $\M$. Then the curvature tensor $K$
has the form
\begin{equation}\label{KR7-DT=0}
\begin{split}
&K(x,y,z,w) = R(x, y, z, w)\\%
&\phantom{K(x,y,z,w) =}
-\eta(y)\eta(z)g\left(\n_x \xi,\n_w \xi\right)%
+\eta(x)\eta(z)g\left(\n_y \xi,\n_w \xi\right)\\
&\phantom{K(x,y,z,w) =}
-\eta(x)\eta(w)g\left(\n_y \xi,\n_z \xi\right)%
 +\eta(y)\eta(w)g\left(\n_x \xi,\n_z \xi\right)\\
&+ 2 \left(\n_x \eta\right)y\left(\n_z \eta\right)w+\left(\n_y
\eta\right)z\left(\n_x \eta\right)w+\left(\n_z
\eta\right)x\left(\n_y \eta\right)w.
\end{split}
\end{equation}
\end{prop}

\begin{proof}
The equality \eqref{KR7-DT=0} is a consequence of \eqref{KR-DT=0},
\eqref{T7} and the following calculations
\[
\begin{split}
g\left(T(x,y),T(z,w)\right)=&4\left\{\left(\n_x
\eta\right)y\left(\n_z \eta\right)w\right.\\
&\left.\phantom{\left\{\right.}
-\eta(y)\eta(z)g\left(\n_x \xi,\n_w \xi\right)%
+\eta(x)\eta(z)g\left(\n_y \xi,\n_w \xi\right)\right.\\
&\left.\phantom{\left\{\right.}
-\eta(x)\eta(w)g\left(\n_y \xi,\n_z \xi\right)%
 +\eta(y)\eta(w)g\left(\n_x \xi,\n_z \xi\right)
 \right\},
\end{split}
\]
\begin{equation}\label{sgTT}
\begin{split}
\sx
\left\{g\left(T(x,y),T(z,w)\right)\right\}&=4\sx\left\{\left(\n_x
\eta\right)y\left(\n_z \eta\right)w\right\}
\\%
&=\frac{1}{2}\left(\D\eta\wedge \D\eta\right)(x,y,z,w).
\end{split}
\end{equation}
\end{proof}
\begin{cor}\label{prop-KR7-sK=DT=0}
Let the $\f$KT-connection $D$ have a $D$-parallel torsion tensor
$T$ on $\M\in\F_7$. If the curvature tensor of $D$ is of
$\f$-K\"ahler type then
\begin{equation}\label{KR7-sK=DT=0}
\begin{split}
&K(x,y,z,w) = R(x, y, z, w)+ \left(\n_x \eta\right)y\left(\n_z
\eta\right)w
\\%
&\phantom{K(x,y,z,w) =}
-\eta(y)\eta(z)g\left(\n_x \xi,\n_w \xi\right)%
+\eta(x)\eta(z)g\left(\n_y \xi,\n_w \xi\right)\\
&\phantom{K(x,y,z,w) =}
-\eta(x)\eta(w)g\left(\n_y \xi,\n_z \xi\right)%
 +\eta(y)\eta(w)g\left(\n_x \xi,\n_z \xi\right).
\end{split}
\end{equation}
\end{cor}
\begin{proof}
Expression \eqref{KR7-sK=DT=0} follows from \eqref{KR7-DT=0},
\eqref{sgTT} and \eqref{sgTT=0}.
\end{proof}

The following statement is obtained from \propref{prop-Rf37} using
\eqref{T7}.
\begin{cor}
Let the $\f$KT-connection $D$ have a curvature tensor of
$\f$-K\"ahler type and a $D$-parallel torsion tensor $T$ on
$\M\in\F_7$. The curvature tensor $R$ has property \eqref{R7xi}
and the following property:
\[
\begin{split}
R(x,y,\f z,\f w)=&- R(x, y, z, w)\\%
&+\eta(y)\eta(z)g\left(\n_x \xi,\n_w \xi\right)%
-\eta(x)\eta(z)g\left(\n_y \xi,\n_w \xi\right)\\
&
+\eta(x)\eta(w)g\left(\n_y \xi,\n_z \xi\right)%
-\eta(y)\eta(w)g\left(\n_x \xi,\n_z \xi\right).
\end{split}
\]
\end{cor}

\section{A Lie group as a 5-dimensional $\F_7$-manifold}\label{sec:5}

Let $G$ be a 5-dimensional real connected Lie group and let
$\mathfrak{g}$ be its Lie algebra. If $\left\{E_i\right\}$
$(i=1,2,\dots,5)$ is a global basis of left-invariant vector
fields of $G$, we define an invariant almost contact structure
$(\f,\xi,\eta)$ and a left invariant B-metric $g$ on $G$ as
follows:
\begin{equation}\label{f}
\begin{array}{c}
\f E_1 = E_3,\quad \f E_2 = E_4,\quad \f E_3 =-E_1,\quad \f E_4 =
-E_2,\quad \f E_5 =0;\\%
\xi=E_5;\qquad \eta(E_i)=0\; (i=1,2,3,4),\quad \eta(E_5)=1;
\end{array}
\end{equation}
\begin{equation}\label{g}
\begin{split}
g(E_1,E_1)= & g(E_2,E_2)=-g(E_3,E_3)=-g(E_4,E_4)=g(E_5,E_5)=1,
\\%
&g(E_i,E_j)=0,\; i\neq j,\quad  i,j\in\{1,2,3,4,5\}.
\end{split}
\end{equation}

Thus, because of \eqref{str}, the induced 5-dimensional manifold
$\Lf$ is an almost contact B-metric manifold.

By analogy with the non-Abelian almost complex structure on an
even-dimensio\-nal Lie group we introduce the following notion. An
almost contact structure 
on a Lie group $G$ is called \emph{non-Abelian} if the following
property holds with respect to the Lie algebra $\mathfrak{g}$
\begin{equation}\label{case}
[\f X,\f Y]=-[X,Y].
\end{equation}

Obviously, the last equation implies $[\xi,Y]=0$ which is a
sufficient condition the vector field $\xi$ to be Killing with
respect to $g$.

\begin{prop}\label{prop-F37}
Let $\Lf$ be an almost contact B-metric manifold with a
non-Abelian almost contact structure on the Lie group $G$. Then
the mani\-fold  $\Lf$ belongs to the class $\F_3\oplus\F_7$.
\end{prop}
\begin{proof}
The known property of the Levi-Civita connection $\n$ of  $g$ on
$\Lf$
\begin{equation}\label{LC}
2g(\n_XY,Z)=g([X,Y],Z)+g([Z,X],Y)+g([Z,Y],X)
\end{equation}
and the equivalent condition $[\f X,Y]=[X,\f Y]$ of \eqref{case}
imply
\begin{equation}\label{2F}
2F(X,Y,Z)=g\bigl([X,\f Y]-\f[X,Y],Z\bigr)+g\bigl([X,\f
Z]-\f[X,Z],Y\bigr).
\end{equation}
Therefore, using \eqref{case}, we obtain
\[
\sx F(X,Y,Z)=\sx g\bigl([X,\f Y]-[\f X,Y],Z\bigr)=0.
\]
Hence and since $\xi$ is Killing, because of \eqref{case},  we
establish that $\Lf$ belongs to $\F_3\oplus\F_7$.
\end{proof}

\begin{cor}\label{cor-370}
Let $\Lf$ be an almost contact B-metric manifold with non-Abelian
almost contact structure on the Lie group $G$. Then the necessary
and sufficient conditions $\Lf$ to belong to the component classes
of $\F_3\oplus\F_7$ are respectively:
\[
\begin{array}{c}
\F_3: \eta\left([X,Y]\right)=0;\;\;\; \F_7: \f[\f
X,Y]=\f^2[X,Y];\;\; \F_0: [X,Y]=-\f[\f X,Y].
\end{array}
\]
\end{cor}
\begin{proof}
From \propref{prop-F37} and \eqref{2F}, using \eqref{1.4},
\eqref{1.4a} and $\F_0=\F_3\cap\F_7$, we obtain the above
conditions.
\end{proof}

Now, let us consider the Lie algebra $\mathfrak{g}$ on $G$,
determined by the following non-zero commutators:
\begin{equation}\label{com2}
\begin{split}
&[E_1,E_2]=-[E_3,E_4]=-\lm_1E_1-\lm_2E_2+\lm_3E_3+\lm_4E_4+2\mu_1E_5,
\\%
&[E_1,E_4]=-[E_2,E_3]=-\lm_3E_1-\lm_4E_2-\lm_1E_3-\lm_2E_4+2\mu_2E_5,
\end{split}
\end{equation}
where $\lm_i, \mu_j \in \R$ $(i=1,2,3,4; j=1,2)$
. 

Then we prove the following
\begin{thm}
Let $(G,\f,\xi,\eta,g)$ be the almost contact B-metric manifold,
determined by \eqref{f}, \eqref{g} and \eqref{com2}. Then it
belongs to the class $\F_7$.
\end{thm}

\begin{proof}
We check directly that the commutators in \eqref{com2} satisfy the
condition for $\F_7$ from \corref{cor-370}. Therefore, we
establish that the corresponding manifold $\Lf$ belongs to the
class $\F_7$. It is not a $\F_0$-manifold if and only if
$\mu_1^2+\mu_2^2\neq 0$ holds.
\end{proof}

Using \eqref{LC} and \eqref{com2}, we obtain the components of the
Levi-Civita connection on the manifold:
\begin{equation}\label{nabli}
\begin{array}{l}
\n_{E_1}E_1=-\n_{E_3}E_3=\lm_1E_2-\lm_3E_4, \\%
\n_{E_1}E_2=-\n_{E_3}E_4=-\lm_1E_1+\lm_3E_3+\mu_1E_5, \\%
\n_{E_1}E_3=\n_{E_3}E_1=\lm_3E_2+\lm_1E_4, \\%
\n_{E_1}E_4=\n_{E_3}E_2=-\lm_3E_1-\lm_1E_3+\mu_2E_5, \\%
\n_{E_2}E_1=-\n_{E_4}E_3=\lm_2E_2-\lm_4E_4-\mu_1E_5, \\%
\n_{E_2}E_2=-\n_{E_4}E_4=-\lm_2E_1+\lm_4E_3, \\%
\n_{E_2}E_3=\n_{E_4}E_1=\lm_4E_2+\lm_2E_4-\mu_2E_5,\\%
\n_{E_2}E_4=\n_{E_4}E_2=-\lm_4E_1-\lm_2E_3,\\%
\n_{E_1}E_5 =\n_{E_5}E_1 =-\mu_1E_2+\mu_2E_4, \\%
\n_{E_2}E_5 =\n_{E_5}E_2 =\mu_1E_1-\mu_2E_3,\\%
\n_{E_3}E_5=\n_{E_5 }E_3=-\mu_2E_2-\mu_1E_4,\\%
\n_{E_4}E_5=\n_{E_5 }E_4=\mu_2E_1+\mu_1E_3, \\%
\n_{E_5}E_5 =0.
\end{array}
\end{equation}

Let us consider the $\f$KT-connection $D$ on $\Lf$ defined by
\eqref{1}, \eqref{TQ} and \eqref{T7}. Then, by \eqref{nabli} we
compute its components as follows:
\begin{equation}\label{F7D}
\begin{array}{ll}
D_{E_1}E_1=-D_{E_3}E_3=\lm_1E_2-\lm_3E_4, \; & 
D_{E_1}E_2=-D_{E_3}E_4=-\lm_1E_1+\lm_3E_3,\\%
D_{E_1}E_3=D_{E_3}E_1=\lm_3E_2+\lm_1E_4, \; & 
D_{E_1}E_4=D_{E_3}E_2=-\lm_3E_1-\lm_1E_3, \\%
D_{E_2}E_1=-D_{E_4}E_3=\lm_2E_2-\lm_4E_4, \; & 
D_{E_2}E_2=-D_{E_4}E_4=-\lm_2E_1+\lm_4E_3, \\%
D_{E_2}E_3=D_{E_4}E_1=\lm_4E_2+\lm_2E_4,\; & 
D_{E_2}E_4=D_{E_4}E_2=-\lm_4E_1-\lm_2E_3, \\%
D_{E_5}E_1 =-2\mu_1E_2+2\mu_2E_4, \; & D_{E_5}E_2 =2\mu_1E_1-2\mu_2E_3,\\%
D_{E_5}E_3=-2\mu_2E_2-2\mu_1E_4, \; &
D_{E_5}E_4=2\mu_2E_1+2\mu_1E_3, \\%
D_{E_i}E_5=0,\ (i=1,2,3,4,5).
\end{array}
\end{equation}
The basic non-zero components of the torsion of $D$ are only:
\[
\begin{array}{l}
T(E_1,E_2,E_5)=T(E_3,E_4,E_5)=2\mu_1,\\%
T(E_2,E_3,E_5)=T(E_4,E_1,E_5)=2\mu_2.
\end{array}
\]
Hence, using \eqref{F7D}, we obtain that the corresponding
components of the covariant derivative of $T$ with respect to $D$
are zero. Thus, we prove the following
\begin{prop}
The $\f$KT-connection $D$ with components \eqref{F7D} on the
$\F_7$-manifold $\Lf$ has a $D$-parallel torsion $T$.
\end{prop}


According to \eqref{g} and \eqref{nabli}, the curvature tensor $R$
has the following non-zero components
$R_{ijkl}=R(E_{i},E_{j},E_{k},E_{l})$:
\begin{equation}\label{Rijkl}
\begin{split}
R_{1212}&=R_{3434}
=\left(\lm_1^2+\lm_2^2-\lm_3^2-\lm_4^2\right)+3\mu_1^2,
\\%
R_{1234}&=R_{3412}
=-\left(\lm_1^2+\lm_2^2-\lm_3^2-\lm_4^2\right)-2\mu_1^2+\mu_2^2,
\\%
R_{1414}&=R_{2323}
=-\left(\lm_1^2+\lm_2^2-\lm_3^2-\lm_4^2\right)+3\mu_2^2,
\\%
R_{1423}&=R_{2314}
=\left(\lm_1^2+\lm_2^2-\lm_3^2-\lm_4^2\right)+\mu_1^2-2\mu_2^2,
\\%
R_{1214}&=-R_{1223}=-R_{2312}=R_{2334}=R_{1412}=-R_{1434}\\%
&=-R_{3414}=R_{3423}
=2\left(\lm_1\lm_3+\lm_2\lm_4\right)+3\mu_1\mu_2,\\%
R_{1324}&=R_{2413}=-\left(\mu_1^2+\mu_2^2\right),\qquad R_{1535}=R_{2545}=-2\mu_1\mu_2,\\%
R_{1515}&=R_{2525}=-R_{3535}=-R_{4545}=-\mu_1^2+\mu_2^2,
\end{split}
\end{equation}
and the remaining ones are determined by \eqref{Rijkl} and
properties \eqref{1.10} and \eqref{1.10B}.

Next we obtain the following basic components
$K_{ijkl}=K(E_{i},E_{j},E_{k},E_{l})$ of the curvature tensor $K$
of $D$, using \eqref{F7D}:
\begin{equation}\label{Kijkl}
\begin{split}
K_{1212}&=-K_{1234}=-K_{3412}=K_{3434}
=\left(\lm_1^2+\lm_2^2-\lm_3^2-\lm_4^2\right)+4\mu_1^2,
\\%
K_{1414}&=-K_{1423}=-K_{2314}=K_{2323}
=-\left(\lm_1^2+\lm_2^2-\lm_3^2-\lm_4^2\right)+4\mu_2^2,
\\%
K_{1214}&=-K_{1223}=-K_{2312}=K_{2334}=K_{1412}=-K_{1434}\\%
&=-K_{3414}=K_{3423}
=2\left(\lm_1\lm_3+\lm_2\lm_4\right)+4\mu_1\mu_2.
\end{split}
\end{equation}

Hence, we establish immediately that $K$ is a $\f$-K\"ahler-type
curvature tensor if and only if $\mu_1=\mu_2=0$, which is
equivalent to the trivial case when $\Lf$ is a $\F_0$-manifold.
Bearing in mind \thmref{thm-K-sK=DT=0}, we establish that $\Lf$
has a \emph{weak $\f$KT-structure} (i.e. $\D T\neq 0$).

Using \eqref{Rijkl} and \eqref{Kijkl}, we compute the
corresponding components of the Ricci tensor $\rho$ and  $\rho^D$
and the values of the scalar curvatures $\tau$ and $\tau^D$ for
the connections $\n$ and $D$, respectively:
\begin{equation}\label{rho}
\begin{split}
&\rho_{11}=\rho_{22}=-\rho_{33}=-\rho_{44}=
-2(\lm_1^2+\lm_2^2-\lm_3^2-\lm_4^2)-2(\mu_1^2-\mu_2^2),\\%
&\rho_{13}=\rho_{24}=-4\left(\lm_1\lm_3+\lm_2\lm_4\right)-4\mu_1\mu_2,
\qquad
\rho_{55}=4(\mu_1^2-\mu_2^2), \\%
 &\rho_{11}^D=\rho_{22}^D=-\rho_{33}^D=-\rho_{44}^D=
-2(\lm_1^2+\lm_2^2-\lm_3^2-\lm_4^2)-4(\mu_1^2-\mu_2^2),\\%
&\rho_{13}^D=\rho_{24}^D=-4\left(\lm_1\lm_3+\lm_2\lm_4\right)-8\mu_1\mu_2;
\end{split}
\end{equation}

\[
\begin{split}
&\tau= -8(\lm_1^2+\lm_2^2-\lm_3^2-\lm_4^2)-4(\mu_1^2-\mu_2^2), \\%
&\tau^D=-8(\lm_1^2+\lm_2^2-\lm_3^2-\lm_4^2)-16(\mu_1^2-\mu_2^2).
\end{split}
\]


Bearing in mind the value of the square norm of $\n \f$ on $\Lf$
\[
\nf=-8\left(\mu_1^2-\mu_2^2\right),
\]
we obtain that the constructed manifold is an
isotropic-$\F_0$-manifold if and only if $\mu_1=\pm\mu_2$.

Finally, we summarize the latter characteristics in the following
\begin{prop}
Let $\Lf$ be the $\F_7$-manifold  determined by \eqref{f},
\eqref{g} and \eqref{com2} and $D$ be the $\f$KT-connection
defined by \eqref{T7}. The following conditions are equivalent:
\begin{enumerate}[(i)]
    \item The manifold $\Lf$ is an isotropic-$\F_0$-manifold;
    \item The scalar curvatures for $\n$ and $D$ are equal;
    \item The vectors $\n_{E_i}\xi$ $(i=1,2,3,4)$ are isotropic;
    \item The equality $\mu_1=\pm\mu_2$ is valid.
\end{enumerate}
\end{prop}

According to \eqref{g} and \eqref{rho} we obtain
\begin{prop}
The $\F_7$-manifold $\Lf$, determined by \eqref{f}, \eqref{g} and
\eqref{com2}, is Einsteinian if and only if the following
conditions are valid
\[
\mu_1\mu_2
=-\left(\lm_1\lm_3+\lm_2\lm_4\right),\qquad 
\mu_1^2-\mu_2^2
=-\frac{1}{3}( \lm_1^2+\lm_2^2-\lm_3^2-\lm_4^2).
\]
\end{prop}

\section*{Acknowledgments}
The author wishes to thank Stefan Ivanov and Kostadin Gribachev
for several useful discussions about this work; Galia Nakova for
the comments about the construction of the commutators in
\eqref{com2}.

This work was partially supported by project NI11-FMI-004 of the
Scientific Research Fund at the University of Plovdiv.



\begin{thebibliography}{9}

\bibitem{BarDotVerb}
M. L. Barberis, I. Dotti and M. Verbitsky, Canonical bundles of
complex nilmanifolds with applications to hypercomplex geometry,
{\it Math. Res. Lett.} {\bf 16} (2009) 331--347.

\bibitem{BarFi}
M. L. Barberis and A. Fino, New strong HKT manifolds arising from
quaternionic representations, \emph{Math. Z.}  \textbf{267} (2011)
717--735.

\bibitem{Bis}
J.-M. Bismut, A local index theorem for non-K\"{a}hler manifolds,
\emph{Math. Ann.} \textbf{284} (1989) 681--699.

\bibitem{FiGran}
A. Fino and G. Grantcharov, Properties of manifolds with
skew-symmetric torsion and special holonomy, {\it Adv. Math.} {\bf
189} (2004) 439--450.

\bibitem{Fri-Iv2}
T. Friedrich and S. Ivanov, Parallel spinors and connections with
skew-sym\-met\-ric torsion in string theory, \emph{Asian J. Math.}
\textbf{6} (2002) 303--336.

\bibitem{Fri-Iv}
T. Friedrich and S. Ivanov, Almost contact manifolds, connections
with torsion, and parallel spinors, \emph{J. Reine Angew. Math.}
\textbf{559} (2003) 217--236.

\bibitem{GaMiGr}
G. Ganchev, V. Mihova and K. Gribachev, Almost contact manifolds
with B-metric, \emph{Math. Balkanica (N.S.)} \textbf{7} (1993)
261--276.

\bibitem{GHR}
S. J. Gates, C. M. Hull and M. Ro\v{c}ek, Twisted multiplets and
new supersymmetric non-linear $\sigma$-models, \emph{Nucl. Phys.
B} \textbf{248} (1984) 157--186.

\bibitem{Gau}
P. Gaudochon, Hermitian connections and Dirac operators,
\emph{Boll. Un. Mat. Ital. B} (7), \textbf{11} (1997) 257--288.

\bibitem{GranPoon}
G. Grantcharov and Y. S. Poon, Geometry of hyper-K\"ahler
connections with torsion, {\it Commun. Math. Phys.} {\bf 213}
(2000) 19--37.

\bibitem{GriManDim12}
K. Gribachev, M. Manev and S. Dimiev, On the almost hypercomplex
pseudo-Her\-mit\-ian manifolds, in \emph{Trends of Complex
An\-al\-ysis, Differential Geometry and Mathematical Physics},
eds. S. Dimiev and K. Sekigawa (World Sci. Publ., Singapore, River
Edge, 2003) pp. 51--62.

\bibitem{GriMan24}
K. Gribachev and M. Manev, Almost hypercomplex pseudo-Her\-mitian
manifolds and a 4-dimensional Lie group with such structure,
\emph{J. Geom.}  \textbf{88} (2008) 41--52.

\bibitem{HoPa}
P. S. Howe and G. Papadopoulos, Twistor spaces for hyper-K\"ahler
manifolds with torsion, \emph{Phys. Lett. B} \textbf{379} (1996)
80--86.

\bibitem{IvPapa}
S. Ivanov and G. Papadopoulos, Vanishing theorems and string
backgrounds, \emph{Classical Quant. Grav.} \textbf{18} (2001)
1089--1110.

\bibitem{KoNo}
S. Kobayshi and K. Nomizu, \emph{Foundations of differential
geometry} (Intersc. Publ., New York, 1969).

\bibitem{Man}
M. Manev, On the conformal geometry of almost contact manifolds
with B-metric, Ph.D. thesis, Plovdiv University, 1998.

\bibitem{Man28}
M. Manev, A connection with parallel torsion on almost
hypercomplex ma\-ni\-folds with Hermitian and anti-Hermitian
metrics. \emph{J. Geom. Phys.} \textbf{61}  (2011) 248--259.

\bibitem{ManGri32}
M. Manev and K. Gribachev, A connection with parallel totally
skew-sym\-metric torsion on a class of almost hypercomplex
manifolds with Hermitian and anti-Hermitian metrics, \emph{Int. J.
Geom. Methods Mod. Phys.}  \textbf{8} (2011) 115--131.

\bibitem{ManSek18}
M. Manev and K. Sekigawa, Some four-dimensional almost
hypercomplex pseu\-do-Hermitian manifolds, in \emph{Contemporary
Aspects of Complex Analysis, Differential Geometry and
Mathematical Physics}, eds. S. Dimiev and K. Sekigawa (World Sci.
Publ., Hackensack, 2005) pp. 174--186.

\bibitem{Mek2}
D. Mekerov, A connection with skew symmetric torsion and K\"ahler
curva\-ture tensor on quasi-K\"ahler manifolds with Norden metric,
\emph{C. R. Acad. Bulgare Sci.} \textbf{61} (2008) 1249--1256.

\bibitem{Mek6}
D. Mekerov, Connection with parallel totally skew-symmetric
torsion on almost complex manifolds with Norden metric, \emph{C.
R. Acad. Bulgare Sci.} \textbf{62} (2009)  1501--1508.


\bibitem{Mek7}
D. Mekerov, On the geometry of the connection with totally
skew-sym\-metric torsion on almost complex manifolds with Norden
metric, \emph{C. R. Acad. Bulgare Sci.} \textbf{63} (2010) 19--28.

\bibitem{Stro}
A. Strominger, Superstrings with torsion, \emph{Nucl. Phys. B}
\textbf{274} (1986) 253--284.

\bibitem{Sw}
A. Swann, Twisting Hermitian and hypercomplex geometries,
\emph{Duke Math. J.}  \textbf{155} (2010) 403--431.

\end{thebibliography}
\end{document}